\documentclass[reqno]{amsart}

\numberwithin{equation}{section}
\usepackage{latexsym}
\usepackage{amsmath}
\usepackage{amssymb}
\usepackage{mathrsfs}
\usepackage{graphicx}
\usepackage{ifthen}
\usepackage[T1]{fontenc}
\usepackage[latin1]{inputenc}

\numberwithin{equation}{section}

\newtheorem{defi}{Definition}[section]
\newtheorem{theorem}[defi]{Theorem}

\newtheorem{corollary}[defi]{Corollary}

\newtheorem{remark}[defi]{Remark}

\usepackage{latexsym}
\usepackage{amsmath}
\usepackage{amssymb}

\newcommand{\cH}{{\mathcal H}}
\newcommand{\cB}{{\mathcal B}}

\newcommand{\EE}{{\mathbb E}}

\newcommand{\NN}{{\mathbb N}}
\newcommand{\R}{{\mathbb R}}
\newcommand{\RR}{{\mathbb R}}

\renewcommand{\epsilon}{\varepsilon}

\newcommand{\nn}{\nonumber}

\begin{document}

\title[Addendum]
{Addendum to the Paper\\
"On Quasilinear Parabolic Evolution Equations in Weighted $L_p$-Spaces II"}

\author{Jan Pr\"uss}
\address{Martin-Luther-Universit\"at Halle-Witten\-berg\\
         Institut f\"ur Mathematik \\
         Theodor-Lieser-Strasse 5\\
         06120 Halle, Germany}
\email{jan.pruess@mathematik.uni-halle.de}

\author{Mathias Wilke}
\address{Universit\"at Regensburg, Fakult\"at f\"ur Mathematik, 93040 Regensburg, Germany}

\email{mathias.wilke@ur.de}

\maketitle

\noindent
This note is devoted to a small, but essential, extension of Theorem 2.1 of our recent paper \cite{LPW14}. The improvement is explained in Section 1 and proved in Section 2. The importance of the extension is demonstrated in Section 3 with an application to the Navier-Stokes system in critical $L_q$-spaces.

\section{The Improvement}
Let $X_0,X_1$ be Banach spaces such that $X_1$ embeds densely in $X_0$, let $p\in(1,\infty)$ and  $1/p<\mu\leq1$.
As in \cite{LPW14}, we consider the following quasilinear parabolic evolution equation
\begin{equation}\label{qpp}
\dot{u} +A(u)u = F_1(u)+F_2(u),\; t>0,\quad u(0)=u_1.
\end{equation}
The space of initial data will be the real interpolation space $X_{\gamma,\mu} =(X_0,X_1)_{\mu-1/p,p}$, and the state space of the problem is $X_\gamma=X_{\gamma,1}$. Let $ V_\mu\subset X_{\gamma,\mu}$ be open and $u_1\in V_\mu$. Furthermore, let $X_\beta=(X_0,X_1)_\beta$ denote  the \emph{complex interpolation spaces}. We will impose the following assumptions.

\medskip

\noindent
{\bf (H1)} $(A,F_1)\in C^{1-}(V_\mu; \cB(X_1,X_0)\times X_0)$.

\medskip

\noindent
{\bf (H2)} $F_2: V_\mu\cap X_\beta \to X_0$ satisfies the estimate
$$ |F_2(u_1)-F_2(u_2)|_{X_0} \leq C \sum_{j=1}^m (1+|u_1|_{X_\beta}^{\rho_j}+|u_2|_{X_\beta}^{\rho_j})|u_1-u_2|_{X_{\beta_j}},$$
for some numbers $m\in\NN$, $\rho_j\geq0$, $\beta\in(\mu-1/p,1)$, $\beta_j\in[\mu-1/p, \beta]$, where $C$ denotes a constant which may depend on $|u_i|_{X_{\gamma,\mu}}$.

\medskip

\noindent
{\bf (H3)} For all $j=1,\ldots,m$ we have
$$ \rho_j( \beta-(\mu-1/p)) + (\beta_j -(\mu-1/p)) \leq 1 -(\mu-1/p).$$

\noindent
Allowing for equality in {\bf (H3)} is not for free, there is no free lunch, in particular not in mathematics. We have to impose additionally the following structural {\bf Condition (S)} on the Banach spaces $X_0$ and $X_1$.

\medskip

\noindent
{\bf (S)} The space $X_0$ is of class UMD. The embedding
$$ {H}^{1}_p(\RR;X_0)\cap L_{p}(\RR;X_1)\hookrightarrow {H}^{1-\beta}_{p}(\RR;X_\beta),$$
is valid for each $\beta\in (0,1)$, $p\in(1,\infty)$.

\begin{remark}\mbox{}
\begin{enumerate}
\item By the {\em Mixed Derivative Theorem}, Condition {\bf (S)} is valid if $X_0$ is of class UMD, and if there is an operator $A_\#\in \cH^\infty(X_0)$, with domain
${\sf D}(A_\#)=X_1$, and $\cH^\infty$-angle $\phi_{A_\#}^\infty<\pi/2$. We refer to Pr\"uss and Simonett \cite{PrSi16}, Section 4.5.
\item The assumption that $X_0$ is a UMD space in Condition {\bf (S)} cannot be skipped, since the maximal domain of $\left(\frac{d}{dt}\right)^\alpha$ in $L_p(\R;X_0)$ is given by
$H_p^\alpha(\R;X_0)$ if $X_0$ is a UMD space.
\item The Condition {\bf (S)} implies the embedding
$${_0\mathbb{E}}_{1,\mu}(0,T):={_0H}^1_{p,\mu}((0,T);X_0)\cap L_{p,\mu}((0,T); X_1)\hookrightarrow {_0H}_{p,\mu}^{1-\beta}((0,T);X_\beta).$$
Indeed, if $u\in {_0\mathbb{E}}_{1,\mu}(0,T)$, then we first extend $u$ to a function $\tilde{u}\in {_0\mathbb{E}}_{1,\mu}(\R_+)$ by \cite[Lemma 2.5]{MeySchn12}. This in turn is equivalent to the fact 
$$[t\mapsto t^{1-\mu}\tilde{u}(t)]\in {_0H}^1_{p}(\R_+;X_0)\cap L_{p}(\R_+; X_1).$$
In a next step we extend $v(t):=t^{1-\mu}\tilde{u}(t)$ by zero to $\R_-$ to obtain 
$$\tilde{v}\in {H}^1_{p}(\R;X_0)\cap L_{p}(\R; X_1).$$
By Condition {\bf (S)} it follows that $\tilde{v}\in {H}^{1-\beta}_{p}(\RR;X_\beta)$, hence $v\in {_0H}^{1-\beta}_{p}(\RR_+;X_\beta)$ and therefore $u\in {_0H}^{1-\beta}_{p,\mu}((0,T);X_\beta)$.
\end{enumerate}
\end{remark}

\noindent
The announced extension of Theorem 2.1 of \cite{LPW14} is the following result.

\begin{theorem}\label{main} Suppose that the structural assumption {\bf (S)} holds, and assume that hypotheses {\bf (H1), (H2), (H3)} are valid. Fix any $u_0\in V_\mu$ such that $A_0:= A(u_0)$ has maximal $L_p$-regularity. Then there  is $T=T(u_0)>0$ and $\varepsilon =\varepsilon(u_0)>0$ with $\bar{B}_{X_{\gamma,\mu}}(u_0,\varepsilon) \subset V_\mu$ such that problem \eqref{qpp} admits a unique solution
$$ u(\cdot, u_1)\in H^1_{p,\mu}((0,T);X_0)\cap L_{p,\mu}((0,T); X_1) \cap C([0,T]; V_\mu),$$
for each initial value $u_1\in \bar{B}_{X_{\gamma,\mu}}(u_0,\varepsilon).$ There is a constant $c= c(u_0)>0$ such that
$$ |u(\cdot,u_1)-u(\cdot,u_2)|_{\EE_{1,\mu}(0,T)} \leq c|u_1-u_2|_{X_{\gamma,\mu}},$$
for all $u_1,u_2\in \bar{B}_{X_{\gamma,\mu}}(u_0,\varepsilon)$.
\end{theorem}

\noindent
We call $j$ subcritical if in {\bf (H3)} strict inequality holds, and critical otherwise. As $\beta_j\leq \beta<1$, any $j$ with $\rho_j=0$ is subcritical.
Furthermore,  {\bf (H3)} is equivalent to $ \rho_j\beta+\beta_j -1\leq \rho_j(\mu-1/p)$, hence the minimal value of $\mu$ is given by
$$ \mu_{crit} = \frac{1}{p} + \beta -\min_j(1-\beta_j)/\rho_j.$$
This number defines the {\em critical weight}. Theorem \ref{main} shows that we have local well-posedness of \eqref{qpp} for initial values in
the space $X_{\gamma,\mu_{crit}}$, provided {\bf (H1)} holds for $\mu=\mu_{crit}$. Therefore it makes sense to name this space the {\em critical space} for \eqref{qpp}.

The main difference of Theorem \ref{main} to our previous result, Theorem 2.1 in \cite{LPW14}, is that here we may allow for equality in Condition {\bf (H3)}, at the expense of assuming {\bf (S)}.
In the applications presented in \cite{LPW14}, there was no need for this equality, as strict inequality had to be imposed to ensure {\bf (H1)}. But meanwhile we realized that equality in {\bf (H3)} is an important issue. To demonstrate this, we use Theorem \ref{main} to study the Navier-Stokes equations, and refer also to the recent paper Pr\"uss \cite{Pru16} for an application to the quasi-geostrophic equations on compact surfaces without boundary in $\RR^3$. We mention that the proofs of the remaining results in \cite{LPW14} remain valid without any changes.

\section{Proof of the Main Result}
We show how to extend the proof of Theorem 2.1 in \cite{LPW14} to the case of equality in {\bf (H3)}. For this purpose we fix any critical index $j$, i.e.
$$\rho_j(\beta-(\mu-1/p)) + \beta_j -(\mu-1/p) = 1 -(\mu-1/p),$$
and set
$$ \frac{1}{r} =\frac{\beta_j-(\mu-1/p)}{ 1-(\mu-1/p)},\quad \frac{1}{r^\prime} = \rho_j\frac{\beta-(\mu-1/p)}{1-(\mu-1/p)}.$$
Then we have $1/r<1$, $1/r^\prime <\rho_j$, and $1/r+1/r^\prime =1$. Using the notation in the proof of Theorem 2.1 in \cite{LPW14}, we start with equation (2.12) in \cite{LPW14}:
\begin{multline*}
|F_2(v)-F_2(u_0^*)|_{\mathbb{E}_{0,\mu}(0,T)}\\
\le C_{\varepsilon_0}\sum_{j=1}^m\left(\int_0^T
(1+|v(t)|_{X_{\beta}}^{\rho_j}+|u_0^*(t)|_{X_{\beta}}^{\rho_j})^p
|v(t)-u_0^*(t)|_{X_{\beta_j}}^p
t^{(1-\mu)p}dt\right)^{1/p}.
\end{multline*}
We apply H\"olders inequality to the result
\begin{multline*}
|F_2(v)-F_2(u_0^*)|_{\mathbb{E}_{0,\mu}(0,T)}\\
\le C_{\varepsilon_0}\sum_{j=1}^m\left(
(\kappa(T)+|v|_{L_{\rho_j pr',\sigma'}(X_\beta)}^{\rho_j}+|u_0^*|_{L_{\rho_j pr',\sigma'}(X_\beta)}^{\rho_j})
|v-u_0^*|_{L_{pr,\sigma}(X_{\beta_j})}\right),
\end{multline*}
where 
$$\kappa(T):=\frac{1}{(p(1-\mu)+1)^{1/(pr')}}T^{(p(1-\mu)+1)/(pr')}\to 0,$$
as $T\to 0$ and $1-\mu = r(1-\sigma) = \rho_j r^\prime(1-\sigma^\prime)$. Note that $\sigma,\sigma^\prime$ are admissible, as
$$ \sigma = 1-\frac{1}{r} + \frac{\mu}{r}> \frac{1}{pr},\quad \sigma^\prime = 1-\frac{1}{\rho_j r^\prime} +\frac{\mu}{\rho_j r^\prime} > \frac{1}{ \rho_j p r^\prime}.$$
Next we have by Condition {\bf (S)}, Sobolev embedding and Hardy's inequality
\begin{multline*} 
{_0 \EE}_{1,\mu}(0,T) \hookrightarrow {_0H}^{1-\beta_j}_{p,\mu}((0,T); X_{\beta_j})\hookrightarrow {_0H}^{1-\beta_j-\frac{1}{pr'}}_{pr,\mu}((0,T); X_{\beta_j})\\
\hookrightarrow L_{pr,\sigma}((0,T); X_{\beta_j}),
\end{multline*}
as $1/r+1/r'=1$ and
$$ 1-\beta_j -\frac{1}{p} -(1-\mu) = - \frac{1}{pr} -(1-\sigma),$$
see e.g.\ Pr\"uss and Simonett \cite{PrSi16}, Sections 4.5.5 and 3.4.6 or Meyries and Schnaubelt \cite{MeySchn12}. We emphasize that the embedding constants do not depend on $T>0$.  In the same way we obtain the embedding
$$ {_0 \EE}_{1,\mu}(0,T) \hookrightarrow {_0H}^{1-\beta}_{p,\mu}((0,T); X_{\beta})\hookrightarrow L_{\rho_jpr^\prime,\sigma^\prime}((0,T); X_{\beta}),$$
as
$$ 1-\beta -\frac{1}{p} -(1-\mu) = - \frac{1}{\rho_jpr^\prime} -(1-\sigma^\prime) .$$
The triangle inequality first yields
$$|v|_{L_{\rho_j pr',\sigma'}(X_\beta)}\le |v-u_1^*|_{L_{\rho_j pr',\sigma'}(X_\beta)}+|u_1^*|_{L_{\rho_j pr',\sigma'}(X_\beta)},$$
where $u_1^*(t)= e^{-A_0t} u_1$. This implies with $v(0)=u_1$ the estimates
\begin{multline*}
|v-u_1^*|_{L_{\rho_j pr',\sigma'}(X_\beta)}\le C|v-u_1^*|_{_0\mathbb{E}_{1,\mu}}\\
\le C(|v-u_0^*|_{\mathbb{E}_{1,\mu}}
+|u_0^*-u_1^*|_{\mathbb{E}_{1,\mu}})\le C(r+|u_0-u_1|_{X_{\gamma,\mu}})
\end{multline*}
and, by Proposition 3.4.3 of Pr\"uss and Simonett \cite{PrSi16},
\begin{multline*}
|u_1^*|_{L_{\rho_j pr',\sigma'}(X_\beta)}\le |u_0^*-u_1^*|_{L_{\rho_j pr',\sigma'}(X_\beta)}+|u_0^*|_{L_{\rho_j pr',\sigma'}(X_\beta)}\\
\le C(|u_0-u_1|_{X_{\gamma,\mu}}+\varphi(T)),
\end{multline*}
with $\varphi(T) = |u_0^*|_{L_{\rho_jpr^\prime,\sigma^\prime}((0,T); X_{\beta})}$.
Moreover, it holds that
\begin{multline*}
|v-u_0^*|_{L_{pr,\sigma}(X_{\beta_j})}\le |v-u_1^*|_{L_{pr,\sigma}(X_{\beta_j})}+|u_1^*-u_0^*|_{L_{pr,\sigma}(X_{\beta_j})}\\
\le C(r+|u_0-u_1|_{X_{\gamma,\mu}}).
\end{multline*}
Therefore, choosing $T>0$, $r>0$ and $|u_0-u_1|_{X_{\gamma,\mu}}$ small enough, we obtain the estimate $|F_2(v)|_{\mathbb{E}_{0,\mu}}\le r/3$. In fact, $\varphi(T)\to 0$ for $T\to0$, as $u_0^*\in L_{\rho_jpr^\prime,\sigma^\prime}((0,T); X_{\beta})$ by Proposition 3.4.3 of Pr\"uss and Simonett \cite{PrSi16}.

A similar argument applies to the contraction estimate 
\begin{multline*}
|F_2(v_1)-F_2(v_2)|_{\mathbb{E}_{0,\mu}(0,T)}\\
\le C_{\varepsilon_0}\sum_{j=1}^m\left(
(\kappa(T)+|v_1|_{L_{\rho_j pr',\sigma'}(X_\beta)}^{\rho_j}+|v_2|_{L_{\rho_j pr',\sigma'}(X_\beta)}^{\rho_j})
|v_1-v_2|_{L_{pr,\sigma}(X_{\beta_j})}\right),
\end{multline*}
making use of $|v_i|\le |v_i-u_i^*|+|u_i^*|$ and
$$|v_1-v_2|_{L_{pr,\sigma}(X_{\beta_j})}\le |v_1-v_2-(u_1^*-u_2^*)|_{L_{pr,\sigma}(X_{\beta_j})}+|u_1^*-u_2^*|_{L_{pr,\sigma}(X_{\beta_j})}.$$

\section{Application to the Navier-Stokes Equations}
Let $\Omega\subset\RR^n$, $n\geq2$, be a bounded domain with boundary $\partial\Omega$ of class $C^{3-}$, and consider the Navier-Stokes problem
\begin{align}\label{NS}
\partial_t u +u\cdot\nabla u-\Delta u +\nabla \pi &=0,\quad \mbox{in }\; \Omega,\nn\\
{\rm div}\, u &=0,\quad \mbox{in }\; \Omega,\\
u&=0,\quad \mbox{on }\; \partial\Omega,\nn\\
u(0)&=u_0,\quad \mbox{in }\; \Omega.\nn
\end{align}
Here $u$ denotes the velocity field and $\pi$ the pressure. We consider this problem in $L_q(\Omega)^n$ with $1<q<\infty$. Employing the Helmholtz projection $P$ and the Stokes operator $A$, this problem can be reformulated as the abstract semilinear evolution equation
\begin{equation}\label{ANS}
\dot{u} + Au = F(u),\; t>0,\quad u(0)=u_0,
\end{equation}
 in the Banach space
$X_0 =L_{q,\sigma}(\Omega)=P L_q(\Omega)^n$, with bilinear nonlinearity $F$ defined by
\begin{equation}\label{F}
F(u) = G(u,u),\quad G(u_1,u_2)=-P (u_1\cdot\nabla) u_2.
\end{equation}
It is well-known, see e.g.\ Hieber and Saal \cite{HiSa16} or Amann \cite{Ama00}, that the Stokes operator $A=-P\Delta$ with domain
$$ X_1 = {\sf D}(A) := \{ u\in H^2_q(\Omega)^n\cap L_{q,\sigma}(\Omega):\; u=0 \mbox{ on } \partial\Omega\}$$
is sectorial, and admits a bounded $\cH^\infty$-calculus with $\cH^\infty$-angle equal to zero. Therefore, $A$ has maximal $L_p$-regularity, even on the halfline, as 0 belongs to the resolvent set of $A$, and so $-A$ generates an exponentially stable analytic  $C_0$-semigroup in $X_0$. Thus with $A(u)=A$, $F_1=0$, and $F_2=F$, problem \eqref{ANS} is of the form \eqref{qpp}, and Conditions {\bf (H1)} and {\bf (S)} are valid. To apply Theorem \ref{main} we therefore have to estimate the nonlinearity in a proper way.
This will be done as follows. As $P$ is bounded in $L_q(\Omega)^n$, by H\"older's inequality we obtain
$$ |G(u_1,u_2)|_{X_0} \leq C |u_1\cdot\nabla u_2|_{L_q}\leq C |u_1|_{L_{qr^\prime}}|u_2|_{H^1_{qr}},$$
where $r,r^\prime>1$ and $1/r+1/r^\prime=1$. We choose $r$ in such a way that the Sobolev indices of the spaces $L_{qr^\prime}(\Omega)$ and $H^1_{qr}(\Omega)$ are equal, which means
$$ 1-\frac{n}{qr} = -\frac{n}{qr^\prime},\quad \mbox{equivalently}\quad \frac{n}{qr} = \frac{1}{2}(1+\frac{n}{q}).$$
This is feasible if $q\in (1,n)$, we assume this in the sequel. Next we employ Sobolev embeddings to obtain
$$ X_\beta\subset H^{2\beta}_q(\Omega)^n\hookrightarrow L_{qr^\prime}(\Omega)^n\cap H^1_{qr}(\Omega)^n.$$
This requires for the Sobolev index $2\beta-n/q$ of $H^{2\beta}_q(\Omega)$
$$ 1-\frac{n}{qr} = 2\beta-\frac{n}{q},\quad \mbox{i.e.} \quad \beta = \frac{1}{4}( 1+ \frac{n}{q}).$$
The condition $\beta<1$ is equivalent to $n/q<3$, we assume this below. To meet Conditions {\bf (H2), (H3)}, we set $m=1$, $\rho_1=1$, $\beta_1=\beta$. Then {\bf (H3)} requires
$$ 2\beta \leq 1+\mu-1/p, $$
hence the optimal choice $\mu=\mu_{crit}$ is
$$ \mu_{crit} -\frac{1}{p} = 2\beta-1 = \frac{1}{2}(\frac{n}{q}-1).$$
Finally, the constraint $\mu\leq 1$ requires that $p$ should be chosen large enough, to be subject to the condition
$$ \frac{2}{p} +\frac{n}{q}\leq 3.$$
Computing the space of admissible initial values then leads to
$$ X_{\gamma,\mu} = {_0B}^{n/q-1}_{qp}(\Omega)^n\cap L_{q,\sigma}(\Omega).$$
Applying Theorem \ref{main}, this yields the following result on local well-posedness of the Navier-Stokes system \eqref{NS} for initial values in these critical spaces.

\begin{theorem}
Let $q\in (1,n)$, $p\in (1,\infty)$ be such that $2/p+n/q\leq3$, and  suppose that $\Omega\subset\RR^n$ is bounded domain of class $C^{3-}$.

Then, for each initial value $u_0\in {_0B}^{n/q-1}_{qp}(\Omega)^n\cap L_{q,\sigma}(\Omega)$, the Navier-Stokes problem \eqref{NS} admits a unique strong solution $u$ in the class
$$ u \in H^1_{p,\mu}((0,T); L_{q,\sigma}(\Omega))\cap L_{p,\mu}((0,T); H^2_q(\Omega)^n),$$
for some $T=T(u_0)>0$, with $\mu = 1/p+ n/2q -1/2$. The solution exists on a maximal interval $(0,t_+(u_0))$ and depends continuously on $u_0$.
In addition, we have
$$ u \in C([0,t_+); B^{n/q-1}_{qp}(\Omega)^n\cap L_{q,\sigma}(\Omega))\cap C((0,t_+); B^{2(1-1/p)}_{qp}(\Omega)^n),
$$
i.e.\ it regularizes instantly if $2/p +n/q<3$.
\end{theorem}

\noindent
It is an easy consequence of this result that we also have global existence for initial values which are small in one of the critical spaces.

\begin{corollary} Let the assumptions of Theorem 3.1 be valid. Then there exists $r>0$ such that the solution from Theorem 3.1 exists globally, provided $|u_0|_{B^{n/q-1}_{qp}}<r.$
\end{corollary} 
\begin{proof}
By the estimate of $F(u)$ given above, it is easy to show via maximal regularity that $v(t) =u(t) -e^{-At}u_0$ satisfies
$$ |v|_{\EE_{1,\mu}(0,T)} \leq C_1 |u_0|_{X_{\gamma,\mu}}^2 + C_2 |v|_{\EE_{1,\mu}(0,T)}^2,$$
for each $T<t_+(u_0)$. Here $C_1,C_2>0$ are constants independent of $u_0$ and $T$. This inequality implies boundedness of $|v|_{\EE_{1,\mu}(0,T)}$ on $[0,t_+)$,
hence global existence, provided $|u_0|_{X_{\gamma,\mu}}< r:=1/2\sqrt{C_1C_2}$.
\end{proof}

\begin{remark}\mbox{}
\begin{enumerate}
\item Consider the particular case $n=3$, $p=q=2$. Then we have
$$ X_{\gamma,\mu} = {_0H}^{1/2}_2(\Omega)^3\cap L_{2,\sigma}(\Omega),\quad X_{\gamma} = {_0H}^{1}_2(\Omega)^3\cap L_{2,\sigma}(\Omega),$$
which yields the celebrated Fujita-Kato theorem, proved first in 1962 by means of the famous Fujita-Kato iteration, see \cite{FuKa62}.
\item In the general case, observe that the Sobolev index of the spaces $B^{n/q-1}_{qp}(\Omega)$ equals $-1$, it is independent of $q$. These are the critical spaces for the Navier-Stokes equations in $nD$, their homogeneous versions are known to be scaling invariant in $\Omega=\RR^n$. We refer to Cannone \cite{Can97} for the first results in this direction.


\end{enumerate}
\end{remark}

\medskip

\noindent
In a forthcoming paper, we will extend the range of $q$ to $[n,\infty)$. Thus, by the embeddings 
$$B^{n/q_1-1}_{q_1,p_1}(\Omega)\hookrightarrow B^{n/q_2-1}_{q_2,p_2}(\Omega),\quad 1\leq q_1<q_2<\infty, \; p_1,p_2\in[1,\infty]$$
and $$B^{s}_{q,p_1}(\Omega)\hookrightarrow B^{s}_{q,p_2}(\Omega),\quad 1\leq p_1\leq p_2\leq\infty,$$
and by maximal $L_p$-regularity, this will cover the range $1\leq q<\infty$, $1\leq p\leq\infty$.

\begin{small}

\end{small}

\begin{thebibliography}{99}

\bibitem{Ama00}
H.~Amann.
\newblock On the strong solvability of the Navier-Stokes equations.
\newblock {\em J.~Math.~Fluid Mech.} {\bf 2}, 16--98  (2000)

\bibitem{Can97}
M.~Cannone.
\newblock On a generalization of a theorem of Kato on the Navier-Stokes equations.
\newblock {\em Rev.~Mat.~Iberoamericana} {\bf 13}, 515--541  (1997)

\bibitem{FuKa62} 
H.~Fujita and T.~Kato.
On the non-stationary Navier-Stokes system.
{\em Rend.~Sem.~Mat., Univ.~Padova} {\bf 32}, 243--260 (1962)

\bibitem{HiSa16}
M.~Hieber and J.~Saal.
\newblock The Stokes Equation in the $L_p$-Setting: Wellposedness and Regularity Properties.
\newblock {\em Handbook of Mathematical Analysis in Mechanics of Viscous Fluids}, eds.~ Y.~Giga, A.~Novotny.
\newblock  Springer  to appear (2017).

\bibitem{KPW10}
M.~K{\"o}hne, J.~Pr{\"u}ss, and M.~Wilke.
\newblock On quasilinear parabolic evolution equations in weighted
  {$L_p$}-spaces.
\newblock {\em J. Evol. Equ.} {\bf 10}, 443--463 (2010).


\bibitem{LPW14}
J.~LeCrone, J.~Pr{\"u}ss, and M.~Wilke.
\newblock On quasilinear parabolic evolution equations in weighted
  {$L_p$}-spaces {II}.
\newblock {\em J. Evol. Equ.} {\bf 14}, 509--533 (2014).

\bibitem{MeySchn12} M. Meyries, R. Schnaubelt,
Interpolation, embeddings and traces of anisotropic
fractional Sobolev spaces with temporal weights.
{\it J. Funct. Anal.} {\bf 262}, 1200-1229 (2012).

\bibitem{Pru16}
J.~Pr{\"u}ss.
\newblock On the quasi-geostrophic equations on compact surfaces without boundary in $\RR^3$.
\newblock submitted (2016).


\bibitem{PrSi16}
J.~Pr{\"u}ss and G.~Simonett, \emph{Moving Interfaces and Quasilinear Parabolic Evolution Equations},
 Monographs in Mathematics {\bf 105}, Birkh\"auser, Basel  2016.



\end{thebibliography}
\end{document}